\documentclass[12pt, oneside]{amsart}

\advance \topmargin by -\headheight
\advance \topmargin by -\headsep
\evensidemargin \oddsidemargin
\usepackage{color}
\usepackage{enumerate}
\usepackage{bbm}

\usepackage{graphicx}
\usepackage[mathscr]{euscript}
\usepackage{amsthm,amssymb,amsmath}
\usepackage{subfigure}
\allowdisplaybreaks[4]

\begin{document}
\title{Measure doubling in unimodular locally compact groups and quotients}

\author{Zuxiang Kong}
\address{Department of Mathematics, National University of Singapore, Singapore}
\email{zuxiang\_kong@u.nus.edu}

\author{Fei Peng}
\address{Department of Mathematics, National University of Singapore, Singapore}
\email{pfpf@u.nus.edu}

\author{Chieu-Minh Tran}
\address{Department of Mathematics, National University of Singapore, Singapore}
\email{trancm@nus.edu.sg}


\date{}

\newtheorem{theorem}{Theorem}[section]
\newtheorem{lemma}[theorem]{Lemma}
\newtheorem{corollary}[theorem]{Corollary}
\newtheorem{fact}[theorem]{Fact}
\newtheorem{conjecture}[theorem]{Conjecture}
\numberwithin{equation}{theorem}

\newtheorem{proposition}[theorem]{Proposition}
\theoremstyle{definition}
\newtheorem{remark}[theorem]{Remark}
\newtheorem{definition}[theorem]{Definition}
\newtheorem*{thm:associativity}{Theorem \ref{thm:associativity}}
\newtheorem*{thm:associativity2}{Theorem \ref{thm:associativity2}}
\newtheorem*{thm:associativity3}{Theorem \ref{thm:associativity3}}
\def\tri{\,\triangle\,}

\def\d{\,\mathrm{d}}
\def\BM{\mathrm{BM}}
\def\RR{\mathbb{R}}
\def\ZZ{\mathbb{Z}}
\def\QQ{\mathbb{Q}}
\def\TT{\mathbb{T}}
\def\LL{\mathscr{L}}
\def\MM{\mathcal{M}}
\def\II{\mathcal{I}}
\def\id{\mathrm{id}}
\def\tmu{\tilde{\mu}}
\def\pr{\mathrm{p}}
\newcommand\NN{\mathbb N}
\newcommand{\Case}[2]{\noindent {\bf Case #1:} \emph{#2}}
\newcommand\inner[2]{\langle #1, #2 \rangle}
\newcommand{\ign}[1]{}
\newcommand{\glz}{\mathrm{GL}_2(\ZZ)}
\newcommand{\eq}[1]{\begin{align*}#1\end{align*}}
\def\aL{\mathfrak{l}}
\def\aee{=_\ult}

\renewcommand{\epsilon}{\varepsilon}

\def\ult{\mathcal{U}}
\def \Sot{\mathrm{SO}(3,\RR)}
\def \Sod{\mathrm{SO}(d, \RR)}
\def \Gld{\mathrm{GL}(d, \RR)}

\def\cA{\mathscr{A}}
\def\cB{\mathcal{B}}
\def\cP{\mathcal{P}}
\def\cD{\Sigma}
\def\cAe{\mathscr{A}_{\epsilon}}
\def\cAet{\mathscr{A}_{\epsilon, \theta}}

\def \diam{\mathrm{diam}}

\def\EE{\mathbb{E}}

\def \fL {\mathfrak{l}}

\begin{abstract}
We consider a (possibly discrete) unimodular locally compact group $G$ with Haar measure $\mu_G$, and a compact $A\subseteq G$ of positive measure with $\mu_G(A^2)\leq K\mu_G(A)$. Let $H$ be a closed normal subgroup of G and $\pi: G \rightarrow G/H$ be the quotient map. With the further assumption that $A= A^{-1}$, we show $$\mu_{G/H}(\pi A ^2) \leq K^2 \mu_{G/H}(\pi A).$$ We also demonstrate that $K^2$ cannot be replaced by $(1-\epsilon)K^2$ for any $\epsilon>0$.

In the general case (without $A=A^{-1}$), we show $\mu_{G/H}(\pi A ^2) \leq K^3 \mu_{G/H}(\pi A)$, improving an earlier result by An, Jing, Zhang, and the third author. Moreover, we are able to extract a compact set $B\subseteq A$ with $\mu_G(B)> \mu_G(A)/2$ such that $ \mu_{G/H}(\pi B^2) < 2K \mu_{G/H}(\pi B)$.
\end{abstract}

\maketitle

\section{Introduction}

\subsection{Backgrounds and statement of results}
This paper continues the investigation of doubling inequalities and classifications of small doubling sets. Our focus will be on the relationship between sets with small doubling in nonabelian locally compact groups and their quotients.

For a group $G$ equipped with a notion of size $|\cdot|$ (e.g., cardinality, measure, metric entropy, Banach density, etc.) we are often interested in
\begin{enumerate}
\item obtaining doubling inequality of the form $|A^2| \geq f(|A|)$
\item characterizing $G$ and $A \subseteq G$ with the small doubling  condition $|A^2| \leq g(|A|)$ 
\end{enumerate}
where $f$ and $g$ are suitable functions. Problem (1) and (2) can be seen as the dual of one another, and progress in one often leads to progress in the other. Sometimes, we refer to (1) as ``direct problems'' and (2) as ``inverse problems''.

These problems appear naturally in many branches of mathematics. In particular, results for abelian $G$ play absolutely central roles in additive combinatorics~\cite{Sanderquasi, GMGTMarton, U3vssum} as well as the study of geometric inequalities~\cite{Gardner, FigalliICM}. Our theorems later also apply to abelian $G$, but we will focus more on nonabelian groups where the motivation is more transparent.
Here, some elements of (1) and (2) already appeared implicitly in classical results like the Tit's Alternative \cite{T72},  Margulis' construction of expanders~\cite{M75},  Gromov Theorem on groups of polynomial growth~\cite{Gromov}, Sarnak and Xue's bound on multiplicity of automorphic representations~\cite{SX91}, and  Gowers' quasi-random groups~\cite{G08}. The topic really took off when Bourgain and Gamburd realized its role in understanding random walks on simple groups of Lie types~\cite{BG8Jem, BG9Jem, Helfgott08}; see also~\cite{Green2009ApproximateGA}. This brought applications in constructing expander graphs and in sieve theory~\cite{BG08, BourgainGamburd,  BG12, BGS10, PS16, BGT11} and led to the study of finite approximate groups with further applications in geometry~\cite{T08, Hrushovski, BGT}.

The nonabelian continuous setting of locally compact groups and Lie groups with Haar measure has been intensively studied recently with the development on nonabelian Brunn-Minkowski theory~\cite{JT, JTZ, AJTZ, SO3, machado2024} and the study of product mixing phenomena in compact Lie groups, which has surprising applications to quantum communication~\cite{ellis2024, quantum}. (There are also  other nonabelian settings considered in the last decade, e.g., the metric entropy settings~\cite{BG08Spectral, deSaxce15, BenoistdeSaxce}, and the approximate lattice setting~\cite{BjorklundHartnick, SimonApproximateLattice, Beyondlascar}.)

From our point of view, the main reason to consider the continuous setting is to eventually develop a unified approach to all settings. This is encapsulated by Hrushovski's Lie model theorem~\cite{Hrushovski} and subsequent generalizations~\cite{thesis, MW, Hru20, ArturoUdi, ArturoFrank, KruPillay}. These tell us that problem (2) for any group equipped  with a reasonable notion of size can often be reduced to small doubling in Lie groups. The nonabelian Brunn-Minkowski theory suggests an even finer picture, that small doubling in Lie groups (and thus other groups) arises from a few standard examples in Lie groups of small dimension. 

Toward realizing the above picture, and refining qualitative results  from approximate groups, An, Jing, Zhang, and the third author proved the following in~\cite{AJTZ}: If $G$ is a  connected noncompact Lie group, $G$ is unimodular (left Haar measures are right invariant), $\mu_G$ is a Haar measure of $G$, and $A \subseteq G$ is compact with $\mu_G(A^2)\le K\mu_G(A)$, then $G$ has a compact normal subgroup $H$ such that, with $\pi: G \to G/H$ the quotient map, we have
$$ \dim(G/H) \leq 
\log\lfloor K  \rfloor(\log\lfloor K  \rfloor+1)/2 \quad \text{and} \quad \mu_{G/H}( \pi(A)^2) \leq 32 K^6 \mu_{G/H}( \pi(A)). $$ 
Hence, $A$ is related to a set with larger but still small doubling, namely $\pi(A)$, in the Lie group with small dimension $G/H$.
The first inequality is demonstrated to be sharp in the same paper, but we thought the second inequality could be improved. On the other hand, it is folklore knowledge that $\mu_{G/H}( \pi(A)^2) \leq K \mu_{G/H}( \pi(A))$ cannot be achieved. (Indeed, a recent approach to the Polynomial Freiman Ruzsa Conjecture~\cite{GMGTMarton} replaces doubling with entropic doubling, partly because doubling inequalities are not preserved under quotients.)

Our first result is the sharpened version of the second inequality in the case $A=A^{-1}$. In fact, it is much more general as $G$ needs not be compact or connected (e.g. $G$ can be discrete), and $H$ can be any closed and normal subgroup of $G$.

\begin{theorem}\label{symk2}
Let $G$ be a unimodular group with Haar measure $\mu_G$. Let $A=A^{-1}$ be a compact subset of $G$ of positive measure such that $\mu_G(A^2) \leq K\mu_G (A)$. Let $H$ be a closed and normal subgroup of G and $\pi: G \rightarrow G/H$ be the quotient map. Then, 
$$\mu_{G/H}(\pi A ^2) \leq K^2 \mu_{G/H}(\pi A).$$ Moreover, $K^2$ cannot be replaced by $(1-\epsilon)K^2$ for any $\epsilon>0$ in the above statement.  
\end{theorem}
Without assuming $A=A^{-1}$, our approach can show $\mu_{G/H}(\pi A ^2) \leq K^3 \mu_{G/H}(\pi A)$, and that $\mu_{G/H}(\pi A ^2) \leq K_1K_2 \mu_{G/H}(\pi A)$ when $\mu(A^2)\le K_1\mu(A)$ and $\mu(A^{-1}A)\le K_2\mu(A)$; see Theorem~\ref{thm: main} for details. 
We believe that, even without the $A=A^{-1}$ assumption, $K^2$ (rather than $K^3$) should be the correct bound. If this is true, some new idea is needed.

Our second result moves closer to doubling preservation when looking at a large subset:

\begin{theorem} \label{thm:structure set}
Let $G$ be a unimodular group with Haar measure $\mu_G$. Let $A$ be a compact subset of $G$ of positive measure such that $\mu_G(A^2) \leq K\mu_G (A)$. Let $H$ be a closed normal subgroup of G and $\pi: G \rightarrow G/H$ be the quotient map. Then, there exists a compact $B \subseteq A$ with $\mu_G(B) > \frac{1}{2} \mu_G(A)$ such that $\mu_{G/H}(\pi B^2) < 2K \mu_{G/H}( \pi B)$.   
\end{theorem}

\subsection{Notation and convention} We assume some familiarity with locally compact groups and Lie groups; the reader can refer to Appendixes A and B in~\cite{JTZ} for details.

Unless specified otherwise, $G$ is a unimodular locally compact group, $H$ is a closed normal subgroup of $G$, and $G/H$ is the quotient locally compact group. It is well-known that $H$ and $G/H$ are also unimodular.  We let $\pi: G \to G/H$ denote the quotient map. When there is no confusion, for $A \subseteq G$, we will write $\pi A$ instead of $\pi(A)$ and $\pi A^2$ instead of $\pi(A)^2$.

We will let $\mu_G$, $\mu_H$, and $\mu_{G/H}$ denote the Haar measure on $G$, $H$, and $G/H$ respectively. By scaling we arrange that Fubini's theorem holds; that is, for all continuous $f: G \to \RR$ with compact support, 
$$ \int_G f d \mu_G = \int_{G/H} \left(\int_H f(x^{-1}y) d\mu_H(y) \right) d \mu_{G/H} (xH). $$ 
In particular, if $G$ is isomorphic to the direct product of $H$ and $G/H$, then $\mu_G$ is just the product measure of $\mu_H$ and $\mu_{G/H}$. For a compact group, we assume the Haar measure is normal; that is,  the measure of the whole group is $1$. For a discrete group, we assume that the Haar measure is the counting measure. It can be checked that there is no contradiction with the earlier assumptions in the examples we consider later.

For any subset $A \subseteq G$, we denote $$A^{-1} = \{a^{-1} : a \in A\}$$ and we say that a subset $A$ is {\bf symmetric} if $A = A^{-1}$. For $A, B \subseteq G$, their product set $AB$ is the set $\{ab: a \in A, b\in B\}$. We write $A^2$ for $AA$.

\section{lower bounds}

Throughout this section, $\TT = \RR/\ZZ$ is the one-dimensional torus equipped with the usual Euclidean topology. Hence, its normalized Haar measure is just the usual Lebesgue measure for torus. Let $\glz$ be the group of invertible $2\times 2$ matrices equiped with the discrete topology, so the Haar measure is the counting measure in this case. Finally, we equip the product of two locally compact groups  with their product topology. The condition in the convention section implies the chosen Haar measure is  chosen is the product measure.

The following proposition verifies the sharpness assertion in Theorem~\ref{symk2} by providing an example where the quotient doubling is asymptotically $K^2$.

\begin{proposition}\label{K2exp}
Let $G$ be the product $H \times Q$, where $H$ is any compact group, and $Q$ is the product $\glz\times\TT$. Then for every odd $K\in\ZZ^+$, there exists a compact symmetric subset $A \subseteq G$ with positive measure such that $\mu(A^2) = K\mu(A)$ and $$\mu_{Q}(\pi A^2) = (K^2-2K+2) \mu_{Q} (\pi A),$$ 
where $\pi:G\to Q$ is the projection to $Q$.
\end{proposition}
\begin{proof}
    Suppose $K=2N+1$. We denote by $C$ the image of the usual Cantor set in $\TT$ under the quotient map from $\RR$. Then $C$ is compact with $\mu_\TT(C)=0$,   $C=-C$, and $C+C=\TT$. Let $S=\{2^k:k\in[N]\}$ and observe that $|S-S|=N(N-1)+1$. For all $k\in[N]$, define $$M_k=\begin{pmatrix}0&1\\1&2^k\end{pmatrix}\in\glz.$$ Then, consider $$\MM = \{M_1,\dots,M_N,M_1^{-1},\dots,M_N^{-1}\}\text{ and } \II=\left\{\begin{pmatrix}1&0\\0&1\end{pmatrix}\right\},$$ and note that $|(\II\cup\MM)^2|=4N^2+1$. Now let $A=A_1\sqcup A_2$, where $A_1=H \times \II\times\TT$ and~$A_2=\{1_H\} \times \MM\times C$. Both $A_1$ and $A_2$ are compact and symmetric in $G$, and so is $A$.

    We now perform the measure calculations. \eq{\mu_G(A)&=\mu_G(A_1)=\mu_{\glz}(\II)=|\II|=1.\\
    \mu_G(A^2)&= \mu_G(A^2_1 \cup A_1A_2 \cup A_2A_1)\\
    &= \mu_Q((\II\times\TT)^2\cup (\II\times\TT)(\MM\times C)\cup (\MM\times C)(\II\times\TT))\\
    &=\mu_{\glz}(\II^2\cup \II\MM\cup \MM\II)\\&=1+2N\\&=K.\\
    \mu_{Q}(\pi A)&=\mu_Q((\II\times\TT)\cup(\MM\times C))=\mu_{\glz}(\II)=1.\\
    \mu_{Q}(\pi A^2)&=\mu_Q((\II\cup\MM)^2\times\TT)=\mu_{\glz}((\II\cup\MM)^2)=4N^2+1=K^2-2K+2.}
    Hence we obtain the desired conclusion.
\end{proof}

There are a number of other related constructions which serve other purposes. We describe them in the following remark.

\begin{remark}
    A simpler example is possible if one only wants the main theorem to be sharp up to a constant. One can choose $G = H \times Q$ with $H=\TT$ and $Q$ the additive group $\RR$. One can then use a union of translated copies of the Cantor set, in place of $\MM\times C$, to obtain a slightly weaker quotient doubling of $K^2/(4+o(1))$, rather than $K^2-2K+2$. In this case,~$G$ would be connected and abelian.

    Those examples can be discretized to get examples where the group $G$ is discrete.
\end{remark}

\section{Spillover inequality}

A key ingredient of our proofs is the \textit{Spillover Technique}, which we will now elaborate.
For $\mu_G$-measurable $A \subseteq G$, we define its {\it fiber length function} $f_A: G/H \to \RR^{\geq 0}$ by setting 
$$f_A(gH)=\mu_{H}(g^{-1}A\cap H) \text{ for } g \in G, $$
and call $f_A(gH)$ the {\it fiber length} of $A$ at $gH$. For $t\geq 0$, we denote
$$ \pi A_t= \{gH \in G/H: f_A(gH) \geq t \}. $$
In particular, $\pi A_t$ is a supper level set of $f_A$.
\begin{lemma} \label{lem: level set}
    Suppose $A\subseteq G$ is $\mu_G$-measurable. Then the function $f_A$ is $\mu_{G/H}$-measurable, so  $\pi A_t$ is $\mu_{G/H}$ measurable for all $t\geq 0$. We have
    $$   \mu_G(A) = \int_{t \geq 0} \mu_{G/H} (\pi A_t) dt.$$
\end{lemma}
\begin{proof}
    The first statement can be proven in the same way as in the special case when~$G=\RR^2$ and $H=G/H=\RR$, by induction on complexity of Borel sets and then appealing to regularity. The last statement is a standard fact of measure theory.
\end{proof}

Lemma~\ref{lem: level set} explains why we are interested in  the families $(\pi A_t)_{t\geq 0}$ and $(\pi ({A^2})_t)_{t\geq 0}$, and in relating the latter to the former. To this end, we encounter a technical difficulty. Each~$\pi A_t$~($t\ge0$) is measurable, but its product with another measurable set $X \subseteq G/H$ might not be measurable, even when $X$ is compact. One can overcome this issue in several ways. In this paper, we will approximate $\pi A_t$ in a suitable way by a $\sigma$-compact set $\widetilde{\pi A_t}$ in Lemma~\ref{lemma: modification}. Inner measure was used in~\cite{JT}, which is conceptually the same approach. One can also play the set-theoretic trick by working first in the Solovay-Krivine set-theoretic universe, where every subset of a locally compact groups is measurable, and then invoking the Shoenfield Absoluteness Theorem to deduce the result in general; see Remark~\ref{rem: settheory tricks}. 

\begin{lemma} \label{lemma: modification}
If $A \subseteq G$ is $\mu_G$-measurable and $f_A$ is its fiber length function, then there is a family $(\widetilde{\pi A_t})_{t \geq 0}$ of $\sigma$-compact subsets of $\pi A$ such that the following holds:
    \begin{enumerate}[{\rm (i)}]
        \item $\widetilde{\pi A_0} = \pi A_0 =\pi A$,
        \item The family $(\widetilde{\pi A_t})_{t \geq 0}$ is nonincreasing; that is, $\widetilde{\pi A_s} \supseteq \widetilde{\pi A_t}$ for $0\leq s<t$,
        \item $\widetilde{\pi A_t} \subseteq \pi A_t$. Hence, for $ gH\in \widetilde{\pi A_t}$, we have $f_A(gH) \geq t$,
        \item For $0<s< t$, we have $\pi A_t \setminus \widetilde{\pi A_s}$ is a $\mu_{G/H}$-null set. (We might not have $\pi A_t \setminus \widetilde{\pi A_t}$ is $\mu_{G/H}$-null, and this is a replacement.)
    \end{enumerate}
    Moreover, when $A$ is symmetric, we can arrange that $\widetilde{\pi A_t}$ is symmetric for all $t\geq 0$.
\end{lemma}
\begin{proof}
    For each $q\in\QQ$, use the inner regularity of $\mu_{G/H}$ to choose a $\sigma$-compact set~${S_q \subseteq \pi A_q} $ such that  
$$ \mu_{G/H}( \pi A_q \setminus S_q) =0. $$
Set $\widetilde{\pi A_t} = \bigcup_{q\geq t}  S_q $ for $t > 0$ and $\widetilde{\pi A_0}=\pi A$. We can easily check that the properties (i-iv) are satisfied.

Now assume that $A$ is symmetric. Note that $\pi A_q$ is symmetric, so we can arrange that~$S_q$ is symmetric too by replacing it with $S_q\cap S_q^{-1}$ if necessary. From the construction, we get the symmetry of $\widetilde{\pi A_t}$ for all $t\geq 0$.
\end{proof}

We refer to a family $(\widetilde{\pi A_t})_{t\geq 0}$ satisfying the condition of the above lemma a {\it $\sigma$-compact modification} of $(\pi A_t)_{t\geq 0}$.

We now describe the key inequality used later. This was used  and dubbed the spillover technique in~\cite[Lemma~9.4]{JT}.  The form we use here is slightly different, and this is essential for our purposes.

\begin{lemma}[Spillover Inequality] \label{lem: continuousspillover}
    Suppose $A, B \subseteq G$ are compact, $f_B$ the fiber size function of $B$, and $(B_t)_{t\geq 0}$ a modified superlevel sets family of $f_B$. Then we have $$\mu_G(AB)\geq \int_{t\geq 0}\mu_{G/H}(\pi A \widetilde{\pi B_t})dt$$
and
$$\mu_G(BA)\geq \int_{t\geq 0}\mu_{G/H}(\widetilde{\pi B_t}\pi A )dt.$$
\end{lemma}

\begin{proof}
    We simply claim that $\pi A\widetilde{\pi B_t}\subseteq \pi(AB)_t$ and $\widetilde{\pi B_t}\pi A\subseteq \pi(BA)_t$ for all $t$, upon which the lemma follows immediately from Lemma~\ref{lem: level set}. Let $aH\in \pi A$ and $bH\in \widetilde{\pi B_t}$ be arbitrary. Note that \eq{f_{AB}(abH)&=\mu_H(b^{-1}a^{-1}AB\cap H)\\&\ge\mu_H(b^{-1}B\cap H)\\&=f_B(bH)\\&\ge t.} Hence by definition, $\pi A\widetilde{\pi B_t}\subseteq \pi(AB)_t$, and a similar argument shows that $\widetilde{\pi B_t}\pi A\subseteq \pi(BA)_t$ (note that $\mu_H$ is bi-invariant).
\end{proof}

\section{Upper bounds}

Let $A, B \subseteq G$ be $\sigma$-compact subsets of $G$ with positive measure. The {\it Ruzsa distance} between $A$ and $B$ is defined to be
$$d(A, B) = \log \left(\frac{\mu_G (AB^{-1})}{\sqrt{\mu_G(A) \mu_G(B)}} \right).$$

We collect a number of standard facts which can be found in~\cite{T08}.
\begin{fact}
Let $d$ be defined as above, then we have 
\begin{enumerate}
     \item $d(A, A) \geq 0$ for all $\sigma$-compact $A \subseteq G$. Moreover, $d(A, A) = 0$ if and only if $A$ is a translate of an open subgroup of $G$.
    \item $d(A, B) = d(B, A)$ for all $\sigma$-compact $A, B \subseteq G$.
    \item $d(A, C) \leq d(A, B) + d(B, C)$ for all $\sigma$-compact $A, B, C \subseteq G$.
    \item $d(A, B) = d(aA, Bb)$ for all $\sigma$-compact $A, B \subseteq G$ and $a, b \in G$.
\end{enumerate}

\end{fact}

Our main result in this section has three components. The first is a resolution of Theorem~\ref{symk2}, the second concludes a looser bound without the symmetry assumption, and the third is a more general result.

\begin{theorem} \label{thm: main}
Let $G$ be a unimodular group with Haar measure $\mu_G$, $A\subseteq G$ is $\sigma$-compact with positive measure, and $H$  a closed normal subgroup of G and $\pi: G \rightarrow G/H$ be the quotient map. Then, we will have the following:
\begin{enumerate}[{\rm (i)}]
    \item If $\mu_G(A^2)\leq K \mu_G(A)$ and $A=A^{-1}$, then $\mu_{G/H}(\pi A ^2) \leq K^2 \mu_{G/H}(\pi A)$.
    \item If $\mu_G(A^2)\leq K \mu_G(A)$, then $\mu_{G/H}(\pi A ^2) \leq K^3 \mu_{G/H}(\pi A)$.
    \item If $\mu_G(A^2)\leq K_1 \mu_G(A)$ and $\mu_G(A^{-1}A)\leq K_2 \mu(A)$, then $\mu_{G/H}(\pi A ^2) \leq K_1K_2 \mu_{G/H}(\pi A)$.

\end{enumerate}
\end{theorem}

\begin{proof}
We first note that (i) and (ii) follow quite easily from (iii). If   $\mu_G(A^2)\le K\mu_G(A)$ and~$A=A^{-1}$, then setting $K_1 = K_2 = K$, we get $\mu_G(\pi A ^2) \leq K^2 \mu_G(\pi A)$. If~${\mu_G(A^2)\le K\mu_G(A)}$ but  $A=A^{-1}$ is not assumed,  we still know that, 
by Ruzsa Triangle inequality,
$$d(A^{-1},A^{-1}) \leq d(A^{-1},A) + d(A, A^{-1}) =2\log(\mu_G(A^2)/\mu_G(A))\le \log K^2,$$
so $\mu_G(A^{-1}A)\leq K^{2}\mu_G(A).$
Hence, setting $K_1= K$ and $K_2=K^2$, we can deduce (ii). 

It is therefore sufficient to prove (iii). We will nevertheless still prove (i) right now, as this would give illustration for our techniques in the simplest fashion.

Let $f_A : G/H \rightarrow \mathbb{R},\ gH \mapsto \mu _H (g^{-1} A \cap H)$ be the fiber length function of $A$, let $$\pi A_t:= \{ gH : f_A(gH) \geq t \} \text{ for } t \geq 0, $$ and let $( \widetilde{\pi A_t})_{t \geq 0}$ be a $\sigma$-compact modification of $(A_t)_{t\geq 0}$ obtained using Lemma~\ref{lemma: modification}. (The reader willing to ignore measurability issues can pretend that $\widetilde{\pi A_t} = \pi A_t$ for all $t\geq 0$.)
By the spillover inequality (Lemma~\ref{lem: continuousspillover}), we have
\begin{equation} \label{eq: fromspillover}
    \int_{t\geq 0} \mu_{G/H} (\pi A \widetilde{\pi A_t}) d t \leq \mu_G (A^2) \leq K \mu_G(A).
\end{equation}

We now claim that there exists some $t \geq 0$ such that   
\begin{equation} \label{eq: structure}
\mu_{G/H}(\widetilde{\pi A_t}) \neq 0 \quad \text{and} \quad
   \mu_{G/H} (\pi A \widetilde{\pi A_t}) \leq K \sqrt{\mu_{G/H} (\pi A) \mu_{G/H}( \widetilde{ \pi A_t})}.
\end{equation} 
Suppose not. Then  $\mu_{G/H} (\pi A \widetilde{ \pi A_t}) > K \sqrt{\mu_{G/H} (\pi A) \mu_{G/H}( \widetilde{\pi A_t})}$ for all $t\geq 0$ such that $\mu_{G/H}(\widetilde{\pi A_t})\neq0$. By Lemma~\ref{lem: level set} and conditions (iii) and (iv) of Lemma~\ref{lemma: modification},
\begin{align*} 
0<K\mu_G (A) &= \int_{\substack{t \geq 0\\\mu_{G/H}(\pi A_t)\neq 0}} K \mu_{G/H} (\pi A_t) d t \label{eq:1}\\
&= \int_{\substack{t \geq 0\\\mu_{G/H}(\widetilde{\pi A_t})\neq 0}} K \mu_{G/H} (\widetilde{\pi A_t}) dt 
\\
&\le \int_{\substack{t\geq 0 \\\mu_{G/H}(\widetilde{\pi A_t})\neq 0}} K \sqrt{\mu_{G/H} (\pi A) \mu_{G/H}( \widetilde{\pi A_t})}dt 
\\
&< \int_{\substack{t\geq 0\\\mu_{G/H}(\widetilde{\pi A_t})\neq 0}} \mu_{G/H} (\pi A \widetilde{\pi A_t}) dt 
\\
&\leq \int_{t\geq 0} \mu_{G/H} (\pi A \widetilde{\pi A_t}) dt, 
\end{align*}
which contradicts inequality~\eqref{eq: fromspillover}.

Hence, we can obtain $t\geq 0$ where condition~\eqref{eq: structure} holds. This implies $$d(\pi A,\widetilde{\pi A_t}^{-1}) \leq \log K.$$  Recall that  $\widetilde{\pi A_t}$ is symmetric, as arranged in Lemma~\ref{lemma: modification}, so we have  $d(\pi A, \widetilde{\pi A_t}) \leq \log K$.
Now applying Ruzsa Triangle Inequality we have 
$$d(\pi A, \pi A) \leq d(\pi A, \widetilde{\pi A_t}) + d(\widetilde{\pi A_t}, \pi A) \leq 2\log K = \log K^2.$$
Since $\pi A$ is symmetric, this implies $\mu(\pi A^2) \leq K^2 \mu(\pi A)$.

We now prove (iii).  Let $f_A$, $f_{A^{-1}}$, $\pi A_t$, and $\pi (A^{-1})_t$ be defined as before. 
We note that $$(\pi A_t)^{-1}=\pi (A_t)^{-1}=\{g^{-1}H:f_A(gH)\ge t\}=\{gH:f_{A^{-1}}(gH)\ge t\}=\pi (A^{-1})_t,$$so we can use $\pi A_t^{-1}$ throughout the proof without ambiguities. We will pretend that all the sets involved are measurable.
The fact that we do not necessarily have this is handled in a similar fashion to what we did in (i) (e.g., by replacing $\pi A_t$ with $\widetilde{\pi A_t}$ and $\pi A_t^{-1}$ with~$(\widetilde{\pi A_t})^{-1}$). 
We will skip a number of other details to avoid repetitiveness and drop the group subscript for improved readability as it can be understood from the context.

By the spillover inequality (Lemma~\ref{lem: continuousspillover}), we have the estimates 
\eq{ \int_{t \geq 0} \mu (\pi A \pi A_t) d t &\leq \mu (A^2) \le K_1 \mu(A)\quad\text{and}\\
\int_{t \geq 0} \mu (\pi A_t^{-1}\pi A) d t &\leq \mu (A^{-1}A) \le K_2 \mu(A),}
where both integrals exist by monotonicity. For $t\ge0$, let 
$$g(t) = \sqrt{\mu (\pi A \pi A_t) \mu (\pi A^{-1}_t\pi A)}.$$
By Hölder's inequality,
\begin{equation} \label{eq:fromspillover2}
    \int_{t\geq 0} g(t) d t \leq \sqrt{\int_{t\geq 0} \mu (\pi A \pi A_t) d t \int_{t\geq 0} \mu (\pi A^{-1}_t\pi A) d t}
    \le\sqrt{K_1 K_2} \mu(A).
\end{equation}

We claim that there exists $t\geq 0$ such that $\mu(\pi A_t) \neq 0$ and $g(t) \leq \sqrt{K_1 K_2}\mu(\pi A_t)$. Suppose not; then
$g(t) > \sqrt{K_1 K_2}\mu(\pi A_t)$ for all $t\ge0$ such that $\mu(\pi A_t)\neq0$. By Lemma~\ref{lem: level set},
\begin{align*}
\sqrt{K_1 K_2} \mu(A) = \int_{\substack{t \geq 0\\\mu(\pi A_t)\neq0}} \sqrt{K_1 K_2}\mu(\pi A_t) d t<\int_{\substack{t \geq 0\\\mu(\pi A_t)\neq0}} g(t) d t,
\end{align*}
contradicting inequality~\eqref{eq:fromspillover2}. Hence, we obtain $t$ with $\mu (\pi A_t)\neq 0$ and

$$\sqrt{\mu (\pi A \pi A_t) \mu (\pi A^{-1}_t\pi A)} = g(t) \leq \sqrt{K_1 K_2}\mu(\pi A_t).$$
Taking the natural log on both sides and using Ruzsa Triangle Inequality, we have
\begin{align*}
\log(K_1 K_2) 
& \geq \log \left( \frac{\mu (\pi A \pi A_t)}{\mu(\pi A_t)} \right) + \log \left( \frac{\mu (\pi A_t ^{-1} \pi A)}{\mu(\pi A_t)} \right)\\
& = d(\pi A, \pi A_t ^{-1}) + d(\pi A_t ^{-1}, \pi A^{-1}) \\
& \geq d(\pi A, \pi A^{-1}).
\end{align*}
In other words, $\mu(\pi A ^2) \leq K_1 K_2 \mu(\pi A)$, as desired. 
\end{proof}

 Some readers might think that measurability issues serve nothing but distraction to the central ideas in the above proof. Some may also feel uncomfortable with the use of Axiom of Choice in Lemma~\ref{lemma: modification}. The following remark is for such readers. This remark is also applicable to Theorem~\ref{thm:structure set c}.

\begin{remark} \label{rem: settheory tricks}
Solovay showed assuming ZF (the Zermelo-Fraenkel system of axioms for set theory) and the existence of an inaccessible cardinal that there is a universe of set theory (a model of ZF) where all subsets of $\RR$ are Lebesgue measurable~\cite{Solovay}. The assumption of inaccessible cardinal cannot be removed, but Krivine showed only using ZF that there is a model of ZF where the Lebesgue measure can be extended to an invariant measure on all subsets of reals (without nice properties like inner and outer regularity)~\cite{Krivine}. Even though stated only for subsets of reals, it is well known that their result holds for all Polish space and hence applies also to locally compact groups. In these universes of set theory, the proof of our theorem can proceed without replacing $\pi A_t$ with $\widetilde{\pi A_t}$.\newpage

   Another result in set theory called Shoenfield Absoluteness~\cite{Shoenfield} can be used to show  that Theorem~\ref{thm: main}, if holds in some model of ZF, also holds in others, which establishes our result. To carry that out, one first uses the Gleason--Yamabe theorem (Hilbert 5th Problem)~\cite{Gleason} to reduce the statement to Lie groups. Then one shows by standard techniques that our theorem is sufficiently simple (either $\mathbf{\Pi}^2_1$ or $\mathbf{\Sigma}^2_1$ in the so-called analytical hierarchy). The fact that such statement is true in one model of ZF if and only if it is true in another is the content of Shoenfield Absoluteness Theorem. This technique also shows Theorem~\ref{thm: main} is true regardless of the Axiom of Choice, even though it is used in our proof.

\end{remark}

\section{Subsets with linear doubling in the quotient}

We now prove a slight generalization of Theorem~\ref{thm:structure set}. In particular, we allow $A$ to be $\sigma$-compact, and also allow more flexible constants. By taking $\alpha=2$, this implies Theorem~\ref{thm:structure set}.

\begin{theorem} \label{thm:structure set c}
Let $G$ be a unimodular group with Haar measure $\mu_G$. Let $A$ be a $\sigma$-compact subset of $G$ of positive measure such that $\mu_G(A^2) \leq K\mu_G (A)$. Let $H$ be a closed normal subgroup of G and $\pi: G \rightarrow G/H$ be the quotient map. Then, for all $\alpha >1$, there exists a compact subset $B \subseteq A$ with $\mu_G(B) > \frac{\alpha-1}{\alpha} \mu_G(A)$ such that $\mu_{G/H}(\pi B^2) < \alpha K \mu_{G/H}( \pi B)$.   
\end{theorem}

\begin{proof}
Let $f_A$ be the fiber length function of $A$ and, as before, set $$\pi A_t:= \{ gH : f_A(gH) \geq t \} \text{ for } t \geq 0.$$ 
We will pretend that all the sets involved are measurable.
When this is not true, we handle it using $\sigma$-compact modifications as in the proof of (i) of Theorem~\ref{thm: main}. 
As in the proof of (iii) of Theorem~\ref{thm: main}, we will skip a number of other details to avoid repetitiveness and drop the group subscript  as it can be understood from the context.

By the Spillover Inequality (Lemma~\ref{lem: continuousspillover}), we have
\begin{equation} \label{eq: diagonalspillover}
    \int_{t \geq 0} \mu (\pi A_t \pi A_t)  d t \leq \int_{t \geq 0} \mu (\pi A \pi A_t) d t  \leq \mu (A^2) \leq K \mu(A).
\end{equation}

We claim that there exists a $t \geq 0$ such that $\mu(\pi A_t) \neq 0$ and  $\mu (\pi A_t \pi A_t) \leq K\mu(\pi A_t)$.  Suppose not, then 
we have
\begin{align*} K\mu(A)=\int_{\substack{t \geq 0\\\mu(\pi A_t)\neq 0}} K \mu (\pi A_t) d t 
<\int_{t \geq 0} \mu (\pi A_t \pi A_t) d t,
\end{align*}
a contradiction. In fact, the same argument shows that the set of such $t$ must have positive measure.
Hence, we know that the larger set 
$$S=\{ s \geq 0:    \mu(A_s) \neq 0  \text{ and } \mu(\pi A_s \pi A_s) < \alpha K \mu(\pi A_s) \} $$
is of positive measure.

Consider first the case where $\inf(S) = 0$. For $t \geq 0$, we set 
$$A_t = \pi^{-1} (\pi A_{t}) \cap A.$$
Note that $\sup_{s\in S} \mu(A_s) = \sup_{t>0} \mu(A_t) = \mu(A)$. Hence, one can choose $s \in S$ such that $\mu(A_s)>(\alpha-1)\mu(A)/\alpha$. Since $A_s$ is measurable, one can use the inner regularity property of Haar measure to choose compact $B \subseteq A_s$ such that the desired conditions are satisfied. 

From now on, we assume $\inf S >0$. By inequality~\eqref{eq: diagonalspillover}, we have 
\begin{align*}
K \mu(A) &\geq \int_{0}^{\inf(S)} \mu (\pi A_t \pi A_t) d t + \int_{\inf(S)}^{\infty} \mu (\pi A_t \pi A_t) d t \\
& > \int_{0}^{\inf(S)} \mu (\pi A_t \pi A_t) d t  \\
& \geq \alpha K \int_{0}^{\inf(S)}  \mu (\pi A_t) d t.
\end{align*}
The strictness of the second inequality comes from the fact
that $S$ has positive measure, and every $s \in S$ satisfies $\mu(\pi A_s)>0$ (and hence $\mu(\pi A_s\pi A_s)>0$). 
By continuity of integrals, 
$$\lim_{s \to \inf(S)^+} \int_0^s  \mu (\pi A_t) d t = \int_0^{\inf(S)}  \mu (\pi A_t) d t.$$
Therefore, one can find $s \in S$ such that 
\begin{align*}
    K \mu(A) &> \alpha K \int_{0}^{s}  \mu (\pi A_t) d t  \\ &=\alpha K(\mu(A)-\mu(A_s)+s\mu(\pi A_s))\\&>  \alpha K(\mu(A) - \mu(A_s)),
\end{align*}
where the equality follows from Lemma~\ref{lem: level set}. It follows that $\mu(A_s)> (\alpha-1) \mu(A)/\alpha$ and $\mu(\pi A^2_s) < \alpha K \mu(\pi A_s)$. Finally, we approximate $A_s$ from inside by a compact set using inner regularity to get the desired $B$.
\end{proof}

\section*{Acknowledgements}
The authors would like to thank Jun Le Goh, Yifan Jing, and Dilip Raghavan for their helpful comments and discussions.

\bibliographystyle{amsalpha}
\bibliography{ref}

\end{document}